\documentclass[12pt]{amsart}
\usepackage{amsmath, amssymb}
\usepackage{mathrsfs}
\newcommand{\no}[1]{#1}
\renewcommand{\no}[1]{}
\no{\usepackage{times}\usepackage[subscriptcorrection, slantedGreek, nofontinfo]{mtpro}
\renewcommand{\Delta}{\upDelta}}
\usepackage{color}
\usepackage{enumerate,comment}

\date{\today}
\newtheorem{theorem}{Theorem}[section]
\newtheorem{proposition}{Proposition}[section]
\newtheorem{lemma}{Lemma}[section]

\newtheorem{corollary}{Corollary}[section]

\theoremstyle{remark}
\newtheorem{remark}{Remark}[section]

\newcommand{\bel}{\begin{equation} \label}
\newcommand{\ee}{\end{equation}}

\newcommand{\R}{{\mathbb R}}

\def\phi {\varphi}

\renewcommand{\leq}{\leqslant}
\renewcommand{\geq}{\geqslant}

\def\beq{\begin{equation}}
\def\eeq{\end{equation}}
\newcommand{\bea}{\begin{eqnarray}}
\newcommand{\eea}{\end{eqnarray}}
\newcommand{\beas}{\begin{eqnarray*}}
\newcommand{\eeas}{\end{eqnarray*}}

\providecommand{\abs}[1]{\left\lvert#1\right\rvert}
\providecommand{\norm}[1]{\left\lVert#1\right\rVert}

\numberwithin{equation}{section}

\title[Calder\'on problem]{Lecture on Calder\'on problem}

\author[Yavar Kian]{Yavar Kian}
\address{Aix Marseille Univ, Universit\'e de Toulon, CNRS, CPT, Marseille, France}
\email{yavar.kian@univ-amu.fr}

\begin{document}

\begin{abstract}
We consider  the so called Calder\'on problem which corresponds to the determination of a conductivity appearing in an elliptic equation from boundary measurements. Using several known results we propose a simplified and self contained proof of this result.
\end{abstract}

\maketitle

\tableofcontents


\section{Introduction}

 A. P. Calder\'on published in 1980  the  pioneer contribution \cite{C} which motivated many applications in inverse problems.  The problem considered by  A. P. Calder\'on can be formulated as follows " can one  determine the electrical conductivity of a medium by making voltage and current measurements at the boundary of the medium?".  
This problem is known as the Electrical Impedance Tomography (EIT in short). The work of Calder\'on was motivated by oil prospection. Since then, the EIT problem received a lot of attention among the mathematical community with several applications in  medical imaging and geophysical prospection (see \cite{J,Z}). We refer to \cite{Uh} for an overview of recent development for this problem. In this lecture we consider several results related to this problem.

We start with the first positive answer given to the question raised by Calder\'on in \cite{SU}. Following several arguments borrowed in \cite{Ch,Ha,Ka,Ki,SU} we propose a simple proof of this result which can be considered as an introduction to this field.

We consider also some partial data result associated with this problem corresponding to the same problem with excitation and measurements restricted to some portion of the boundary borrowed from \cite{AU,BU}.

\section{The Calder\'on problem with full data}

Let $\Omega$ be a bounded domain of $\mathbb{R}^n$, $n\geq 3$, with $C^2$ boundary $\partial\Omega$. We denote by $\nu(x)$ the outward unit normal to $\partial\Omega$ computed at $x \in \partial\Omega$. We consider $a\in W^{2,\infty}(\Omega)$ such that there exists $a_0>0$ for which the condition
\bel{a0}
\forall x\in\Omega,\quad a(x)\geq a_0
\ee
is fulfilled. For all $\phi\in H^{\frac{3}{2}}(\partial\Omega)$ we consider the boundary value problem
\bel{eq1}
\left\{
\begin{array}{ll}
-\textrm{div}(a\nabla u)=0  & \mbox{in}\ \Omega ,
\\
u=\phi &\mbox{on}\ \partial\Omega.
\end{array}
\right.
\ee
It is well known that this problem admits a unique solution $u_{a,\phi}\in H^2(\Omega)$. Therefore, we can define the Dirichlet-to-Neumann map (DN map in short) associated with \eqref{eq2} given by
$$\mathcal N_a:H^{\frac{3}{2}}(\partial\Omega)\ni\phi\mapsto a\partial_\nu u_{a,\phi}\in H^{\frac{1}{2}}(\partial\Omega)$$
which is a bounded operator. In this context, the Calder\'on problem corresponds to the determination of the conductivity $a$ from the knowledge of $\mathcal N_a$. We consider here the unique recovery of $a$ from $\mathcal N_a$ which can be stated as follows.

\begin{theorem}
\label{t1}
For $j=1,2$, let $a_j \in W^{2,\infty}(\Omega)$ satisfy \eqref{a0} and the condition
\bel{t1a} a_1(x)=a_2(x),\quad \nabla a_1(x)=\nabla a_2(x),\quad x\in\partial\Omega.\ee
Then we have
\bel{t1b} \mathcal N_{a_1}=\mathcal N_{a_2}\Longrightarrow a_1=a_2.\ee

 \end{theorem}

\begin{remark} Note that applying the result of \cite{KV}, one can remove the condition \eqref{t1a} in Theorem \ref{t1}. In these notes we will not consider the work of  \cite{KV} dealing with  the unique recovery of the conductivity at the boundary from the knowledge of the DN map.
\end{remark}

\subsection{An associated inverse problem}
Let $q\in L^\infty(\Omega)$ take values in $\R$ and consider for each $\phi\in H^{\frac{3}{2}}(\partial\Omega)$  the boundary value problem
\bel{eq2}
\left\{
\begin{array}{ll}
-\Delta v+qv=0  & \mbox{in}\ \Omega ,
\\
v=\phi &\mbox{on}\ \partial\Omega.
\end{array}
\right.
\ee
Assuming that for all $w\in H^1_0(\Omega)\cap H^2(\Omega)$   the implication
\bel{q}-\Delta w+qw=0\Longrightarrow w=0\ee
holds true, one can prove that problem \eqref{eq2} admits a unique solution $v_{q,\phi}\in H^2(\Omega)$.
We define $\mathcal Q$ as the subset of all $q\in L^\infty(\Omega)$ such that for all $w\in H^1_0(\Omega)\cap H^2(\Omega)$ the implication \eqref{q} is fulfilled. For all $q\in\mathcal Q$, we define the Dirichlet-to-Neumann map (DN map in short) associated with \eqref{eq2} given by
$$\Lambda_q:H^{\frac{3}{2}}(\partial\Omega)\ni\phi\mapsto \partial_\nu v_{q,\phi}\in H^{\frac{1}{2}}(\partial\Omega).$$
The Calder\'on problem is strongly connected with the determination of $q$ from  $\Lambda_q$. In terms of uniqueness this result can be stated as follows.

\begin{theorem}
\label{t2}
For $j=1,2$, let $q_j \in \mathcal Q$.
Then we have
\bel{t2a} \Lambda_{q_1}=\Lambda_{q_2}\Longrightarrow q_1=q_2.\ee

 \end{theorem}

Note that using the so called Liouville transform, one can deduce Theorem \ref{t1} from Theorem \ref{t2} in the following way.

\begin{corollary}\label{c1} Assume that the conditions of Theorem \ref{t1} are fulfilled and Theorem \ref{t2} holds true. Then we have \eqref{t1b}. \end{corollary}
\begin{proof} We assume that $\mathcal N_{a_1}=\mathcal N_{a_2}$ holds true and we will prove that $a_1=a_2$.
For $\phi\in H^{\frac{3}{2}}(\partial\Omega)$ and $j=1,2$, fix $u_j\in H^2(\Omega)$ solving
$$
\left\{
\begin{array}{ll}
-\textrm{div}(a_j\nabla u_j)=0  & \mbox{in}\ \Omega ,
\\
u_j=a_1^{-\frac{1}{2}}\phi &\mbox{on}\ \partial\Omega.
\end{array}
\right.$$
Consider $v_j=a_j^{\frac{1}{2}}u_j$. We have
$$\begin{aligned}-\Delta v_j&=-2\nabla\left(a_j^{\frac{1}{2}}\right)\cdot\nabla u_j-a_j^{\frac{1}{2}}\Delta u_j+(-\Delta a_j^{\frac{1}{2}})u_j\\
\ &=a_j^{-\frac{1}{2}}\left[-\nabla a_j\cdot\nabla u_j-a_j\Delta u_j\right]-\left(a_j^{-\frac{1}{2}}\Delta a_j^{\frac{1}{2}}\right)v_j\\
\ &=a_j^{-\frac{1}{2}}\left[-\textrm{div}(a_j\nabla u_j)\right]-\left(a_j^{-\frac{1}{2}}\Delta a_j^{\frac{1}{2}}\right)v_j\\
\ &=-\left(a_j^{-\frac{1}{2}}\Delta a_j^{\frac{1}{2}}\right)v_j.\end{aligned}$$
Thus, for $q_j=a_j^{-\frac{1}{2}}\Delta a_j^{\frac{1}{2}}$, $v_j$ solves
$$-\Delta v_j+q_jv_j=0,\quad \textrm{in}\ \Omega.$$
Moreover, \eqref{t1a} implies that $v_j$ solves
$$\left\{
\begin{array}{ll}
-\Delta v_j+q_jv_j=0  & \mbox{in}\ \Omega ,
\\
v_j=\phi &\mbox{on}\ \partial\Omega.
\end{array}
\right.$$
Note that due to this connection the uniqueness of the solution of \eqref{eq1} with $a=a_j$ implies that $q_j\in\mathcal Q$. In addition, from \eqref{t1a}, for almost every $x\in\partial\Omega$ we get
$$\begin{aligned}\partial_\nu v_1(x)&=a_1^{\frac{1}{2}}\partial_\nu u_1(x)+\partial_\nu \left(a_1^{\frac{1}{2}}\right) u_1(x)\\
\ &=a_1^{-\frac{1}{2}}\mathcal N_{a_1}\phi(x)+\frac{\nu\cdot\nabla a_1\phi(x)}{2a_1^{\frac{1}{2}}(x)}\\
\ &=a_2^{-\frac{1}{2}}\mathcal N_{a_2}\phi(x)+\frac{\nu\cdot\nabla a_2\phi(x)}{2a_2^{\frac{1}{2}}(x)}=\partial_\nu v_2(x).\end{aligned}$$
From this identity we deduce that $\Lambda_{q_1}=\Lambda_{q_2}$ and from Theorem \ref{t2} we get
\bel{c1a} a_1^{-\frac{1}{2}}\Delta a_1^{\frac{1}{2}}=q_1=q_2=a_2^{-\frac{1}{2}}\Delta a_2^{\frac{1}{2}}.\ee
Fix $y=a_1^{\frac{1}{2}}-a_2^{\frac{1}{2}}$. Using \eqref{c1a}, we find 
$$\begin{aligned}\Delta y&=\Delta \left(a_1^{\frac{1}{2}}\right)-\Delta \left(a_2^{\frac{1}{2}}\right)=a_1^{\frac{1}{2}}\left[ a_1^{-\frac{1}{2}}\Delta \left(a_1^{\frac{1}{2}}\right)-a_1^{-\frac{1}{2}}\Delta \left(a_2^{\frac{1}{2}}\right)\right]\\
\ &=a_1^{\frac{1}{2}}\left[ a_2^{-\frac{1}{2}}\Delta \left(a_2^{\frac{1}{2}}\right)-a_1^{-\frac{1}{2}}\Delta \left(a_2^{\frac{1}{2}}\right)\right]=a_2^{-\frac{1}{2}}\Delta \left(a_2^{\frac{1}{2}}\right)[a_1^{\frac{1}{2}}-a_2^{\frac{1}{2}}]=q_2y.\end{aligned}$$
Then, \eqref{t1a} implies
$$\left\{
\begin{array}{ll}
-\Delta y+q_2y=0  & \mbox{in}\ \Omega ,
\\
y=0 &\mbox{on}\ \partial\Omega.
\end{array}
\right.$$
Using the fact that $q_2\in\mathcal Q$, we deduce that $y=0$ which implies that $a_1=a_2$.\end{proof}

In view  of Corollary \ref{c1}, it is clear that Theorem \ref{t1} follows from Theorem \ref{t2}. From now on we will focus our attention on Theorem \ref{t2}.

\subsection{Construction of complex geometric optics solutions}

In this subsection we will construct  solutions of \eqref{eq2} suitably designed for our inverse problem. These solutions, called complex geometric optics solutions, will be the main ingredient in the proof of Theorem \ref{t2}. This approach which has been initiated by \cite{SU} is the main approach considered so far for proving Theorem \ref{t2}. In this section we will use several arguments borrowed from \cite{Ch,Ha,Ka,Ki,SU}.

The complex geometric optics solutions under consideration here will depend on $\xi\in\R^n$. More precisely, for each $\xi\in\R^n$, we fix $\eta_1,\eta_2\in\mathbb S^{n-1}$ such that 
$$\xi\cdot\eta_1=\xi\cdot\eta_2=\eta_1\cdot\eta_2=0.$$
Then, for $\rho>1$ a large parameter, we consider solutions of $-\Delta v_j+q_jv_j=0$ on $\Omega$ taking the form
\bel{CGO1} v_1(x)=e^{\rho \eta_1\cdot x}\left(e^{i\rho \eta_2\cdot x}e^{-i \xi\cdot x}+w_1(x)\right),\ x\in\Omega,\ee
\bel{CGO2} v_2(x)=e^{-\rho \eta_1\cdot x}\left(e^{-i\rho \eta_2\cdot x}+w_2(x)\right),\ x\in\Omega,\ee
with $w_j\in H^2(\Omega)$, $j=1,2$, satisfying
\bel{CGO3}\norm{w_j}_{H^k(\Omega)}=\underset{\rho\to+\infty}{\mathcal O}\left(\rho^{k-1}\right),\ j=1,2,\ k=0,2.\ee
 Here we denote by $H^0(\Omega)$ the space $L^2(\Omega)$. The main point here will be the construction of the expression $w_j$ satisfying \eqref{CGO3}. For this purpose we will use the approach of  \cite{Ha} based on construction of periodic solutions.

\begin{proposition}\label{p1} For $j=1,2$ and for all $\xi\in\R^n$, there exists $\rho_1>0$ such that for all $\rho>\rho_1$  we can find a solution $v_j$ of $-\Delta v_j+q_jv_j=0$ on $\Omega$ taking the form \eqref{CGO1}-\eqref{CGO2} with $w_j\in H^2(\Omega)$ satisfying \eqref{CGO3}.\end{proposition}

\begin{proof} Since the proof of this result is similar for $v_1$ and $v_2$, we will only consider the construction of $v_1$. Note first that
$$\Delta\left(e^{\rho \eta_1\cdot x}e^{i\rho \eta_2\cdot x}\right)=\rho^2\left(|\eta_1|^2-|\eta_2|^2\right)e^{\rho \eta_1\cdot x}e^{i\rho \eta_2\cdot x}=0.$$
Therefore, we have
$$\begin{aligned}&-\Delta v_1\\
&=e^{\rho \eta_1\cdot x}\left(e^{i\rho \eta_2\cdot x}|\xi|^2e^{-i \xi\cdot x}\right)-\Delta\left(e^{\rho \eta_1\cdot x}w_1\right)\\
&=e^{\rho \eta_1\cdot x}\left(e^{i\rho \eta_2\cdot x}|\xi|^2e^{-i \xi\cdot x}\right)+e^{\rho \eta_1\cdot x}(-\Delta w_1-2\rho\eta_1\cdot\nabla w_1-\rho^2w_1).\end{aligned}$$
It follows
$$-\Delta v_1+q_1v_1=0\Longleftrightarrow -\Delta w_1-2\rho\eta_1\cdot\nabla w_1-\rho^2w_1+q_1w_1=-e^{i\rho \eta_2\cdot x}(|\xi|^2+q_1)e^{-i \xi\cdot x}.$$
Using this identity, it seems clear that we need to build $w_1\in H^2(\Omega)$ satisfying 
\bel{p1a} -\Delta w_1-2\rho\eta_1\cdot\nabla w_1-\rho^2w_1+q_1w_1=-e^{i\rho \eta_2\cdot x}(|\xi|^2+q_1)e^{-i \xi\cdot x}\ee
and condition \eqref{CGO3}. For this purpose, we fix $S$ an isometric operator of $\R^n$ such that $S\eta_1=e_1=(1,0,\ldots,0)$ and $R>0$ such that $\Omega\subset Q:=S^*(-R,R)^n$.  Here $S^*$ denotes the adjoint of $S$ with respect to the Euclidean scalar product of $\R^n$. Then, for $F\in L^2(Q)$ we consider a solution $z\in H^2(Q)$ of 
\bel{p1b} -\Delta z-2\rho\eta_1\cdot\nabla z-\rho^2z=F,\quad \textrm{on }Q\ee
taking the form
\bel{p1c}z(x)=\sum_{\alpha\in\mathbb Z^n}a_\alpha\phi_\alpha ,\quad \phi_\alpha(x):=(2R)^{-\frac{n}{2}}e^{\frac{i\pi\eta_1\cdot x}{2R}}e^{\frac{i\pi\alpha\cdot S(x)}{R}}.\ee
The construction of such solutions can be deduced from the following result.
\begin{lemma}\label{l1} For all $F\in L^2(Q)$ the equation \eqref{p1b} admits a solution $z_F\in H^2(Q)$ taking the form \eqref{p1c} and satisfying 
\bel{l1a}\norm{z_F}_{H^k(Q)}\leq C\rho^{k-1}\norm{F}_{L^2(Q)},\quad k=0,2,\ee
with $C$ independent of $F$ and $\rho$.\end{lemma}
The proof of this lemma will be postponed to the end of the present demonstration. Applying Lemma \ref{l1}, we can define the operator
$$\mathcal K_\rho: L^2(Q)\ni F\mapsto z_{F}|_\Omega\in H^2(\Omega)$$
which is bounded. Here $z_F$ denotes a solution of \eqref{p1b} satisfying \eqref{l1a}. We have $\mathcal K_\rho\in \mathcal B(L^2(Q); H^2(\Omega))$ and
\bel{p1d} \norm{\mathcal K_\rho}_{\mathcal B(L^2(Q); H^k(\Omega))}\leq C\rho^{k-1},\quad k=0,2.\ee
Fix $F_\rho\in L^2(Q)$ and $\widetilde{q_1}\in L^\infty(Q)$ defined by
$$F_\rho(x):=\left\{\begin{aligned} -e^{i\rho \eta_2\cdot x}(|\xi|^2+q_1(x))e^{-i \xi\cdot x}\quad &\textrm{if }x\in\Omega\\
0\quad &\textrm{if }x\notin\Omega,\end{aligned}\right.$$
$$\widetilde{q_1}(x):=\left\{\begin{aligned} q_1(x)\quad &\textrm{if }x\in\Omega\\
0\quad &\textrm{if }x\notin\Omega.\end{aligned}\right.$$
Now consider the map
$$\begin{aligned}\mathcal G_\rho:  L^2(\Omega) & \to L^2(\Omega), 
\\ 
w &\mapsto \mathcal K_{\rho}\left[F_\rho-\widetilde{q_1}w\right].\end{aligned}$$
We will prove that there exists $\rho_1>0$ such that  this map admits a fixed point $w_1$ in the closed set $B=\{h\in L^2(\Omega):\ \norm{h}_{L^2(\Omega)}\leq 1\}$ of $L^2(\Omega)$. Indeed, \eqref{p1d} implies
$$\forall y_1,y_2\in B,\quad \norm{\mathcal G_\rho y_1-\mathcal G_\rho y_2}_{L^2(\Omega)}=\norm{\mathcal K_{\rho}\left[\widetilde{q_1}(y_2-y_1)\right]}_{L^2(\Omega)}\leq C\rho^{-1}\norm{q_1}_{L^\infty(\Omega)}\norm{y_1-y_2}_{L^2(\Omega)}$$
and
$$\begin{aligned}\forall y\in B,\quad \norm{\mathcal G_\rho y}_{L^2(\Omega)}&\leq C\rho^{-1}\left[2\norm{q_1}_{L^\infty(\Omega)}+\norm{e^{i\rho \eta_2\cdot x}|\xi|^2e^{-i \xi\cdot x}}_{L^2(\Omega)}\right]\\
&\leq C\rho^{-1}(2\norm{q_1}_{L^\infty(\Omega)}+|\Omega|^{\frac{1}{2}}|\xi|^2).\end{aligned}$$
Consequently, fixing $\rho_1=2C(2\norm{q_1}_{L^\infty(\Omega)}+2\norm{q_2}_{L^\infty(\Omega)}+|\Omega|^{\frac{1}{2}}|\xi|^2)$, we deduce that, for all $\rho>\rho_1$, $\mathcal G_\rho$ is a contraction from $B$ to $B$. Therefore, the Poincar\'e fixed point theorem implies that $\mathcal G_\rho$ admits a unique fixed point $w_1\in B$. Moreover, we have $w_1=\mathcal K_{\rho}\left[F_\rho-\widetilde{q_1}w_1\right]$ which proves that $w_1\in H^2(\Omega)$  solves
$$\begin{aligned}&[-\Delta w_1-2\rho\eta_1\cdot\nabla w_1-\rho^2w_1](x)\\ &=F_\rho(x)-\widetilde{q_1}w_1(x)\\
\ &=-e^{i\rho \eta_2\cdot x}(|\xi|^2+q_1(x))e^{-i \xi\cdot x}-q_1(x)w_1(x),\quad x\in\Omega.\end{aligned}$$
Therefore, $w_1$ solves \eqref{p1a} and \eqref{p1d} implies
$$\norm{w_1}_{H^k(\Omega)}\leq C\rho^{k-1}(2\norm{q_1}_{L^\infty(\Omega)}+|\Omega|^{\frac{1}{2}}|\xi|^2),\ k=0,2,$$
which proves \eqref{CGO3}. \end{proof}

Now that the proof of Proposition \ref{p1} is completed let us turn to Lemma \ref{l1}.

\textbf{Proof of Lemma \ref{l1}.} For all $\alpha\in\mathbb Z^n$, fix
$$\phi_\alpha(x):=(2R)^{-\frac{n}{2}}e^{\frac{i\pi\eta_1\cdot x}{2R}}e^{\frac{i\pi\alpha\cdot Sx}{R}},\quad x\in Q.$$
Combining  the fact that 
$$\phi_\alpha(S^*x)=\psi_\alpha(x)=(2R)^{-\frac{n}{2}}e^{\frac{i\pi\eta_1\cdot S^*x}{2R}}e^{\frac{i\pi\alpha\cdot x}{R}},\quad x\in (-R,R)^n,$$
with the fact that $S$ is an isometry and  $(\psi_\alpha)_{\alpha\in\mathbb Z^d}$ is an orthonormal basis of $L^2((-R,R)^n)$, we deduce that $(\phi_\alpha)_{\alpha\in\mathbb Z^d}$ is an orthonormal basis of $L^2(Q)$. Thus, for all $z$ of the form \eqref{p1c} and $k=0,2$, one can check that
\bel{p1e}\sum_{\alpha\in\mathbb Z^n}(1+|\alpha|^2)^{k}|a_\alpha|^2<+\infty\Longrightarrow z\in H^{k}(Q),\ee
\bel{p1f}z\in H^{k}(Q)\Longrightarrow\norm{z}_{H^k(Q)}^2\leq C\sum_{\alpha\in\mathbb Z^n}(1+|\alpha|^2)^{k}|a_\alpha|^2,\ee
with $C>0$ independent of $z\in H^{k}(Q)$ in the last estimate. Note that  since $S^*e_1=S^{-1}e_1=\eta_1$, one can check that $$
\forall \alpha\in\mathbb Z^n,\quad \eta_1\cdot S^*\alpha\in\mathbb Z$$
and we deduce that
\bel{l11}\forall \alpha\in\mathbb Z^n,\quad \abs{\eta_1\cdot S^*\alpha+\frac{1}{2}}\geq \frac{1}{2}.\ee
Using \eqref{l11}, for $\alpha\in\mathbb Z^n$, we fix
$$a_\alpha:=\frac{\left\langle F,\phi_\alpha\right\rangle_{L^2(Q)}}{\frac{\pi^2}{R^2}\abs{S^*\alpha+\frac{\eta_1}{2}}^2-\rho^2-\frac{2i\pi\rho}{R}\left(S^*\alpha\cdot\eta_1+\frac{1}{2}\right) }.$$
Using the fact that $|S^*\alpha|=|\alpha|$ and \eqref{l11}, we get 
\bel{l12}\begin{aligned}&\abs{\frac{\pi^2}{R^2}\abs{S^*\alpha+\frac{\eta_1}{2}}^2-\rho^2-\frac{2i\pi\rho}{R}\left(S^*\alpha\cdot\eta_1+\frac{1}{2}\right)}\\
&\geq \max\left(\abs{\frac{\pi^2}{R^2}\abs{S^*\alpha+\frac{\eta_1}{2}}^2-\rho^2},\frac{2\pi\rho}{R}\abs{S^*\alpha\cdot\eta_1+\frac{1}{2}}\right)\\
&\geq \max\left(\abs{\frac{\pi^2}{2R^2}\abs{\alpha}^2 -\frac{\pi^2}{4R^2}-\rho^2} , \frac{\pi\rho}{R}\right).\end{aligned}\ee
It follows that
$$\begin{aligned}&\sum_{\alpha\in\mathbb Z^n}(1+|\alpha|^2)^{2}|a_\alpha|^2\\
&\leq\sum_{|\alpha|^2\geq 2+\frac{4R^2\rho^2}{\pi^2}}(1+|\alpha|^2)^{2}|a_\alpha|^2+\sum_{|\alpha|^2\leq 2+\frac{4R^2\rho^2}{\pi^2}}(1+|\alpha|^2)^{2}|a_\alpha|^2\\
&\leq\sum_{|\alpha|^2\geq 2+\frac{4R^2\rho^2}{\pi^2}}\frac{(1+|\alpha|^2)^{2}\abs{\left\langle F,\phi_\alpha\right\rangle_{L^2(Q)}}^2}{\abs{\frac{\pi^2}{2R^2}\abs{\alpha}^2}^2}+\sum_{|\alpha|^2\leq 2+\frac{4R^2\rho^2}{\pi^2}}\frac{(1+|\alpha|^2)^{2}\abs{\left\langle F,\phi_\alpha\right\rangle_{L^2(Q)}}^2}{\frac{\pi^2\rho^2}{R^2}}\\
&\leq\sum_{|\alpha|^2\geq 2+\frac{4R^2\rho^2}{\pi^2}}\frac{16R^4}{\pi^4}\abs{\left\langle F,\phi_\alpha\right\rangle_{L^2(Q)}}^2+\frac{\left(3+\frac{4R^2\rho^2}{\pi^2}\right)^{2}}{\frac{\pi^2\rho^2}{R^2}} \sum_{|\alpha|^2\leq 2+\frac{4R^2\rho^2}{\pi^2}}\abs{\left\langle F,\phi_\alpha\right\rangle_{L^2(Q)}}^2\\
&\leq C\rho^2\sum_{\alpha\in\mathbb Z^n}\abs{\left\langle F,\phi_\alpha\right\rangle_{L^2(Q)}}^2=C\rho^2\norm{F}_{L^2(Q)}^2.\end{aligned}$$
Here $C>0$  depends only on $R$. Combining this estimate with \eqref{p1e}-\eqref{p1f} we deduce that $z\in H^{2}(Q)$ fulfills \eqref{l1a} for $k=2$. In particular, for all $\beta=(\beta_1,\ldots,\beta_n)\in\mathbb N^n$, $|\beta|\leq 2$, one can check that
$$\partial_x^\beta z=\sum_{\alpha\in\mathbb Z^n}\left(\frac{i\pi}{R}\right)^{|\beta|}\left[\prod_{j=1}^n\left(\alpha\cdot Se_j+\frac{\eta_1\cdot e_j}{2}\right)^{\beta_j}\right]a_\alpha\phi_\alpha,$$
where 
$$e_j=(0,\ldots,0,\underbrace{1}_{\textrm{position\ } j},0,\ldots0).$$
It follows 
$$\begin{aligned}-\Delta z&=\sum_{\alpha\in\mathbb Z^n} \frac{\pi^2}{R^2}\left[\sum_{j=1}^n\left(\alpha\cdot Se_j+\frac{\eta_1\cdot e_j}{2}\right)^{2}\right]a_\alpha\phi_\alpha=\sum_{\alpha\in\mathbb Z^n} \frac{\pi^2}{R^2}\left[\sum_{j=1}^n\left(S^*\alpha\cdot e_j+\frac{\eta_1\cdot e_j}{2}\right)^{2}\right]a_\alpha\phi_\alpha\\
\ &=\sum_{\alpha\in\mathbb Z^n} \frac{\pi^2}{R^2}\abs{S^*\alpha +\frac{\eta_1}{2}}^{2}a_\alpha\phi_\alpha.\end{aligned}$$
In the same way, we find 
$$-\eta_1\cdot\nabla z=\sum_{\alpha\in\mathbb Z^n}\frac{-i\pi}{R}\left[S^*\alpha\cdot \eta_1+\frac{1}{2}\right]a_\alpha\phi_\alpha.$$
From these two identities we deduce that
$$\begin{aligned}-\Delta z-2\rho\eta_1\cdot\nabla z-\rho^2z&=\sum_{\alpha\in\mathbb Z^n}\left[\frac{\pi^2}{R^2}\abs{S^*\alpha +\frac{\eta_1}{2}}^{2}-\frac{2i\pi\rho}{R}\left(S^*\alpha\cdot \eta_1+\frac{1}{2}\right)-\rho^2\right]a_\alpha\phi_\alpha
\\
\ &=\sum_{\alpha\in\mathbb Z^n}\left\langle F,\phi_\alpha\right\rangle_{L^2(Q)}\phi_\alpha=F.\end{aligned}$$
Therefore $z$ is a solution of \eqref{p1b}. 
Using \eqref{l12}, we obtain
$$\abs{a_\alpha}\leq \frac{R\abs{\left\langle F,\phi_\alpha\right\rangle_{L^2(Q)}}}{\pi\rho}$$
and, applying \eqref{p1f}, we get
$$\norm{z}_{L^2(Q)}\leq C\norm{F}_{L^2(Q)}\rho^{-1},$$
with $C$ independent of $F$ and $\rho$. This proves \eqref{l1a} for $k=0$ and it completes the proof of the lemma.\qed

\subsection{Proof of Theorem \ref{t2}}

In this section we will complete the proof of Theorem \ref{t2}. For this purpose, let us assume that $\Lambda_{q_1}=\Lambda_{q_2}$. We fix $\xi\in\R^n$ and applying Proposition \ref{p1} we deduce the existence of $v_j\in H^2(\Omega)$, $j=1,2$, solving $-\Delta v_j+q_jv_j=0$  of the form \eqref{CGO1}-\eqref{CGO2} with $w_j\in H^2(\Omega)$ satisfying \eqref{CGO3}. Consider $y_2\in H^2(\Omega)$ solving 
$$\left\{
\begin{array}{ll}
-\Delta y_2+q_2y_2=0  & \mbox{in}\ \Omega ,
\\
y_2=v_1 &\mbox{on}\ \partial\Omega.
\end{array}
\right.$$
Then, $v=y_2-v_1$ solves
$$\left\{
\begin{array}{ll}
-\Delta v+q_2v=(q_1-q_2)v_1  & \mbox{in}\ \Omega ,
\\
v=0 &\mbox{on}\ \partial\Omega.
\end{array}
\right.$$
Moreover, fixing $\phi\in H^{\frac{3}{2}}(\partial\Omega)$ defined by 
$$\phi(x)=v_1(x),\quad x\in\partial\Omega,$$
we find
$$\partial_\nu v=\partial_\nu y_2-\partial_\nu v_1=\Lambda_{q_2}\phi-\Lambda_{q_1}\phi=0.$$
Thus, $v$ satisfies the condition
$$\left\{
\begin{array}{ll}
-\Delta v+q_2v=(q_1-q_2)v_1  & \mbox{in}\ \Omega ,
\\
v=\partial_\nu v=0 &\mbox{on}\ \partial\Omega.
\end{array}
\right.$$
Multiplying, this equation by $v_2$ and integrating by parts, we get
$$\int_\Omega (q_1-q_2)v_1v_2dx =\int_\Omega (-\Delta v+q_2v)v_2dx=\int_\Omega v(-\Delta v_2+q_2v_2)dx=0.$$
We obtain the orthogonality identity
\bel{t1d}\int_\Omega (q_1-q_2)v_1v_2dx =0.\ee
In addition,  \eqref{CGO1}-\eqref{CGO2} imply
$$\int_\Omega (q_1-q_2)v_1v_2dx =\int_\Omega (q_1-q_2)e^{-ix\cdot\xi}dx +\int_\Omega Z_\rho dx,$$
with $$Z_\rho(x)=(q_1(x)-q_2(x))\left[e^{-i\rho\eta_2\cdot x}w_1(x)+e^{i\rho\eta_2\cdot x}e^{-i\xi\cdot x}w_2(x)+w_1(x)w_2(x)\right],\ x\in\Omega.$$
Fixing $q\in L^\infty(\R^n)\cap L^1(\R^n)$ defined by
$$q(x):=\left\{\begin{aligned} q_1(x)-q_2(x),\quad &\textrm{if }x\in\Omega,\\
0,\quad &\textrm{if }x\in\R^n\setminus\Omega\end{aligned}\right.$$
and applying \eqref{CGO3} we get
$$\lim_{\rho\to+\infty}\int_\Omega (q_1-q_2)v_1v_2dx=\int_\Omega qe^{-ix\cdot\xi}dx=\int_{\R^n} qe^{-ix\cdot\xi}dx.$$
Combining this with \eqref{t1d} we obtain
$$\int_{\R^n} qe^{-ix\cdot\xi}dx=0.$$
Since $\xi\in\R^n$ is arbitrary chosen from the injectivity of the Fourier transform we deduce that this condition implies $q=0$ and $q_1=q_2$. This completes the proof of Theorem \ref{t2}.

\section{Partial data results}

Since the pioneer work of \cite{SU}, several authors considered the Calder\'on problem with measurements restricted to some portion of the boundary instead of the full boundary $\partial\Omega$ as stated in Theorem \ref{t1}, \ref{t2}. These class of inverse problems are called inverse problem with partial data. The goal of these inverse problems is to reduce as much as possible the portion of the boundary where the excitation are imposed to the system (corresponding to the support of the Dirichlet input) and the portion where the measurements are made (corresponding to restriction on the knowledge of the Neumann boundary values of solutions). In this section, we consider two class of partial data results associated with the Calder\'on problem. The first class of partial data results corresponds to  partial data results under the assumption of Theorem \ref{t1}, \ref{t2}. The second class of inverse problems requires additional assumption. Namely, they need the knowledge of the coefficients on the neighborhood of the boundary. In this section we give an example of each of these two types of partial data results.

\subsection{Partial data result without restriction } 
In this subsection we prove that the result of Theorem \ref{t1} and \ref{t2} is still true if one restrict the Neumann boundary measurement to, roughly speaking, half of the boundary. This approach which goes back to \cite{BU}, combines the complex geometric optics solutions of Section 3  with Carleman estimates with linear weight. 

Let us first state this result. We start by fixing $\eta\in\mathbb S^{n-1}$ and we consider the decomposition of $\partial\Omega$ into the $\eta$-illuminated face $\partial\Omega_{-,\eta}:=\{x\in\partial\Omega:\ \nu(x)\cdot\eta\leq 0\}$ and the $\eta$-shadowed face $\partial\Omega_{+,\eta}:=\{x\in\partial\Omega:\ \nu(x)\cdot\eta\geq 0\}$. Then, for $\epsilon>0$, we introduce $V:=\{x\in\partial\Omega:\ \nu(x)\cdot\eta<2\epsilon\}$ and we consider the partial DN map
$$\Lambda_q^*:H^{\frac{3}{2}}(\partial\Omega)\ni\phi\mapsto \partial_\nu v_{q,\phi}|_V\in H^{\frac{1}{2}}(V),$$
with $v_{q,\phi}$ solving \eqref{eq2}. The knowledge of this map corresponds to the knowledge of DN map $\Lambda_q$ with restriction of the measurements to $V$ which is a neighborhood of the $\eta$-illuminated face of $\partial\Omega$. The main result of this section can be stated as follows.
\begin{theorem}
\label{t3}
For $j=1,2$, let $q_j \in \mathcal Q$.
Then we have
\bel{t3a} \Lambda_{q_1}^*=\Lambda_{q_2}^*\Longrightarrow q_1=q_2.\ee

 \end{theorem}
Let us also consider the partial DN map associated with \eqref{eq1} given by 
$$\mathcal N_a^*:H^{\frac{3}{2}}(\partial\Omega)\ni\phi\mapsto a\partial_\nu u_{a,\phi}|_V\in H^{\frac{1}{2}}(V).$$
Then in a similar way to Corollary \ref{c1}, we can deduce from Theorem \ref{t3} the following result.
\begin{corollary}\label{c2} Assume that the conditions of Theorem \ref{t1} are fulfilled and Theorem \ref{t3} holds true. Then we have 
$$\mathcal N_{a_1}^*=\mathcal N_{a_2}^*\Longrightarrow a_1=a_2.$$
 \end{corollary}
From now on we will focus our attention on Theorem \ref{t3}. We start with a Carleman estimate which will be the key ingredient for our problem.

\begin{proposition}\label{p2} Let $q\in L^\infty(\Omega)$ and $v\in H^1_0(\Omega)\cap H^2(\Omega)$. Then there exists $C>0$ and $\rho_2>\rho_1$\footnote{ Here $\rho_1$ denotes the constant of Proposition \ref{p1}.} depending only on $\Omega$, $\eta_1$ and $\norm{q}_{L^\infty(\Omega)}$ such that for all $\rho>\rho_2$ we have
\bel{p2a}\begin{aligned}&\rho^2\int_\Omega e^{-2\rho x\cdot\eta_1}|v(x)|^2dx+\rho \int_{\partial\Omega_{+,\eta_1}}e^{-2\rho x\cdot\eta_1}|\partial_\nu v(x)|^2|\eta_1\cdot\nu|d\sigma(x)\\
&\leq C\left[\int_\Omega e^{-2\rho x\cdot\eta_1}|(-\Delta+q)v(x)|^2dx+\rho \int_{\partial\Omega_{-,\eta_1}}e^{-2\rho x\cdot\eta_1}|\partial_\nu v(x)|^2|\eta_1\cdot\nu|d\sigma(x)\right].\end{aligned}\ee
\end{proposition}
\begin{proof} We start by proving \eqref{p2a} for $q=0$. Without lost of generality we may assume that $v$ is real valued. We fix $w(x)=e^{-\rho x\cdot\eta_1}v(x)$ and $P=e^{-\rho x\cdot\eta_1}\Delta e^{\rho x\cdot\eta_1}=\Delta +2\rho \eta_1\cdot\nabla +\rho^2$. Using the fact that $v\in H^1_0(\Omega)$ and the fact that 
$e^{-2\rho x\cdot\eta_1}|\Delta v(x)|^2=|Pw(x)|^2$, we deduce that \eqref{p2a}, with $q=0$, can be deduced from
\bel{p2b}\begin{aligned}&\rho^2\int_\Omega |w(x)|^2dx+\rho \int_{\partial\Omega_{+,\eta_1}}|\partial_\nu w(x)|^2|\eta_1\cdot\nu|d\sigma(x)\\
&\leq C\left[\int_\Omega |Pw(x)|^2dx+\rho \int_{\partial\Omega_{-,\eta_1}}|\partial_\nu w(x)|^2|\eta_1\cdot\nu|d\sigma(x)\right],\end{aligned}\ee
with $C>0$ depending only on $\Omega$. We start with  this last estimate. For this purpose, we decompose the differential operator $P$ into $P_++P_-$
with
$$P_+=\Delta +\rho^2,\quad P_-=2\rho\eta_1\cdot\nabla.$$
We have
\bel{p2c}\begin{aligned}\norm{Pw}_{L^2(\Omega)}^2&\geq \norm{P_-w}_{L^2(\Omega)}^2+2\int_{\Omega}P_+wP_-wdx\\
\ &\geq \norm{P_-w}_{L^2(\Omega)}^2+2I_1+2I_2,\end{aligned}\ee
with 
$$I_1=2\rho \int_{\Omega}\Delta w(\eta_1\cdot\nabla w)dx,\quad I_2=2\rho^3\int_\Omega w(\eta_1\cdot\nabla w)dx.$$
Using the fact that $w_{|\partial\Omega}=0$ we deduce that $\nabla w=(\partial_\nu w)\nu$ on $\partial\Omega$ 
and integrating by parts, we obtain
\bel{p2d}\begin{aligned}I_1&=2\rho\int_{\partial\Omega} \partial_\nu w(\eta_1\cdot\nabla w)d\sigma(x) -2\rho \int_{\Omega}\nabla w\cdot\nabla(\eta_1\cdot\nabla w)dx\\
\ &=2\rho\int_{\partial\Omega} |\partial_\nu w|^2(\eta_1\cdot\nu)d\sigma(x)-\rho \int_{\Omega}\eta_1\cdot\nabla|\nabla w|^2dx\\
\ &=2\rho\int_{\partial\Omega} |\partial_\nu w|^2(\eta_1\cdot\nu)d\sigma(x)-\rho \int_{\Omega}\textrm{div}\left(|\nabla w|^2\eta_1\right)dx\\
\ &=\rho\int_{\partial\Omega} |\partial_\nu w|^2(\eta_1\cdot\nu)d\sigma(x).\end{aligned}\ee
In the same way, we obtain
$$I_2=\rho^3\int_\Omega\textrm{div}\left(w^2\eta_1\right)dx=\rho^3\int_{\partial\Omega} w^2(\eta_1\cdot\nu)d\sigma(x)=0.$$
Combining this with \eqref{p2c}-\eqref{p2d}, we obtain
\bel{p2e}\begin{aligned}&\rho^2\int_\Omega |\eta_1\cdot\nabla w(x)|^2dx+2\rho \int_{\partial\Omega_{+,\eta_1}}|\partial_\nu w(x)|^2|\eta_1\cdot\nu|d\sigma(x)\\
&\leq \int_\Omega |Pw(x)|^2dx+2\rho \int_{\partial\Omega_{-,\eta_1}}|\partial_\nu w(x)|^2|\eta_1\cdot\nu|d\sigma(x).\end{aligned}\ee
In view of this estimate, the proof of \eqref{p2b} will be completed if we can estimate the left hand side of \eqref{p2b} with the one of \eqref{p2e}. This will be possible thanks to the following Poincar\'e type of inequality.
\begin{lemma}\label{l3} There exists a constant $C>0$ depending only on $\Omega$ such that for all $w\in H^1_0(\Omega)$ we have
\bel{l3a} \int_\Omega |w(x)|^2dx\leq C\int_\Omega |\eta_1\cdot\nabla w(x)|^2dx.\ee
\end{lemma}
We postpone the proof of this lemma to the end of the present demonstration. Combining \eqref{l3a} with \eqref{p2e}, we deduce easily \eqref{p2b} and \eqref{p2a} for $q=0$. Now let us consider the case $q\neq0$. For this purpose, note first that 
$$\begin{aligned}\int_\Omega e^{-2\rho x\cdot\eta_1}|(-\Delta+q)v(x)|^2dx&\geq \int_\Omega e^{-2\rho x\cdot\eta_1}\left(\frac{|\Delta v|^2}{2}-|q|^2|v|^2\right)dx\\
\ &\geq \frac{1}{2}\int_\Omega e^{-2\rho x\cdot\eta_1}|\Delta v(x)|^2dx-\norm{q}_{L^\infty(\Omega)}^2\int_\Omega e^{-2\rho x\cdot\eta_1}|v(x)|^2dx.\end{aligned}$$
Combining this with \eqref{p2a} for $q=0$, we obtain
$$\begin{aligned}&\left(\rho^2-2C\norm{q}_{L^\infty(\Omega)}^2\right)\int_\Omega e^{-2\rho x\cdot\eta_1}|v(x)|^2dx+\rho \int_{\partial\Omega_{+,\eta_1}}e^{-2\rho x\cdot\eta_1}|\partial_\nu v(x)|^2|\eta_1\cdot\nu|d\sigma(x)\\
&\leq C\left[2\int_\Omega e^{-2\rho x\cdot\eta_1}|(-\Delta+q)v(x)|^2dx+\rho \int_{\partial\Omega_{-,\eta_1}}e^{-2\rho x\cdot\eta_1}|\partial_\nu v(x)|^2|\eta_1\cdot\nu|d\sigma(x)\right].\end{aligned}$$
Therefore, fixing $\rho_2=2\sqrt{C}\norm{q}_{L^\infty(\Omega)}+\rho_1$ we deduce easily \eqref{p2a} from this estimate.\end{proof}

Now that the proof of Proposition \ref{p2} is completed, let us consider Lemma \ref{l3}.\\
\ \\
\textbf{Proof of Lemma \ref{l3}.} By density we only need to show this result for $w\in\mathcal C^\infty_0(\Omega)$. We fix $R>0$ such that $\Omega\subset B_R:=\{x\in\R^n:\ |x|<R\}$. Fixing $x\in\Omega$ and $h:s\mapsto w[(x-(x\cdot\eta_1)\eta_1)+s\eta_1]$, we deduce that
$$\begin{aligned}w(x)=w[(x-(x\cdot\eta_1)\eta_1)+(x\cdot\eta_1)\eta_1]&=h(x\cdot\eta_1)\\
\ &=\int_{-\infty}^{x\cdot\eta_1}h'(s)ds\\
\ &=\int_{-\infty}^{x\cdot\eta_1} \eta_1\cdot\nabla w[(x-(x\cdot\eta_1)\eta_1)+s\eta_1]ds.\end{aligned}$$
Using the fact that 
$$\forall s\in(-\infty,-R),\quad |(x-(x\cdot\eta_1)\eta_1)+s\eta_1|^2=|x-(x\cdot\eta_1)\eta_1|^2+s^2\geq R^2,$$
and the fact that supp$(w)\subset B_R$, we get
$$w(x)=\int_{-R}^{x\cdot\eta_1} \eta_1\cdot\nabla w[(x-(x\cdot\eta_1)\eta_1)+s\eta_1]ds.$$
Therefore, integrating this last expression with respect to $x\in\R^n$ and making a change of variable with any coordinates having an orthonormal basis containing $\eta_1$, we obtain
$$\int_{\R^n}|w(x)|^2dx=\int_{\eta_1^\bot}\int_{-R}^R |w(x'+t\eta_1)|^2dtdx'=\int_{\eta_1^\bot}\int_{-R}^R\abs{\int_{-R}^{t} \eta_1\cdot\nabla w(x'+s\eta_1)ds}^2dx'dt.$$
Here and from now on, for any $y\in\mathbb R^n$, we denote by $y^\bot$ the set $\{x\in\R^n:\ x\cdot y=0\}$.
Applying H\"older and Cauchy-Schwarz inequality, we obtain 
$$\begin{aligned}\int_{\R^n}|w(x)|^2dx&\leq \int_{\eta_1^\bot}\int_{-R}^R(t+R)\int_{-R}^{t} \abs{\eta_1\cdot\nabla w(x'+s\eta_1)}^2dsdtdx'\\
\ &\leq \int_{\eta_1^\bot}\int_{-R}^R(t+R)\int_{-R}^{R} \abs{\eta_1\cdot\nabla w(x'+s\eta_1)}^2dsdtdx'\\
\ &\leq 4R^2 \int_{\eta_1^\bot}\int_{-R}^{R} \abs{\eta_1\cdot\nabla w(x'+s\eta_1)}^2dsdx'\\
\ &\leq 4R^2\int_{\R^n}\abs{\eta_1\cdot\nabla w(x)}^2dx.\end{aligned}$$
Combining this with the fact that $w$ is supported in $\Omega$ we deduce \eqref{l3a} from this estimate.\qed 

Armed with Proposition \ref{p2}, we are now in position to complete the proof of Theorem \ref{t3}.

\textbf{Proof of Theorem \ref{t3}.} Let us assume that $\Lambda_{q_1}^*=\Lambda_{q_2}^*$. We fix $\eta_1\in\{\theta\in\mathbb S^{n-1}:\ |\eta_1-\eta|<\epsilon\}$ and we consider $\xi\in\eta_1^\bot$, $\eta_2\in\eta_1^\bot\cap\mathbb S^{n-1}$.  Applying Proposition \ref{p1}, for $\rho>\rho_2$,  we deduce the existence of $v_j\in H^2(\Omega)$, $j=1,2$, solving $-\Delta v_j+q_jv_j=0$  of the form \eqref{CGO1}-\eqref{CGO2} with $w_j\in H^2(\Omega)$ satisfying \eqref{CGO3}. Consider $y_2\in H^2(\Omega)$ solving 
$$\left\{
\begin{array}{ll}
-\Delta y_2+q_2y_2=0  & \mbox{in}\ \Omega ,
\\
y_2=v_1 &\mbox{on}\ \partial\Omega.
\end{array}
\right.$$
Then, $v=y_2-v_1$ solves
\bel{t3c}\left\{
\begin{array}{ll}
-\Delta v+q_2v=(q_1-q_2)v_1  & \mbox{in}\ \Omega ,
\\
v=0 &\mbox{on}\ \partial\Omega.
\end{array}
\right.\ee
Moreover, fixing $\phi\in H^{\frac{3}{2}}(\partial\Omega)$ defined by 
$$\phi(x)=v_1(x),\quad x\in\partial\Omega,$$
we find
\bel{t3d}\partial_\nu v=\partial_\nu y_2-\partial_\nu v_1=\Lambda_{q_2}^*\phi-\Lambda_{q_1}^*\phi=0,\quad \textrm{on }V.\ee
Therefore, multiplying \eqref{t3c} by $v_2$ and integrating by parts we get
\bel{t3e}\begin{aligned}\int_\Omega (q_1-q_2)v_1v_2dx &=\int_\Omega (-\Delta v+q_2v)v_2dx\\
\ &=\int_\Omega v(-\Delta v_2+q_2v_2)dx+\int_{U}\partial_\nu vv_2d\sigma(x)\\
\ &=\int_{U}\partial_\nu vv_2d\sigma(x),\end{aligned}\ee
where $$U=\{x\in\partial\Omega:\ \nu(x)\cdot\eta\geq 2\epsilon\}=\partial\Omega\setminus V.$$
Now let us show that
\bel{t3f}\lim_{\rho\to+\infty}\int_{U}\partial_\nu vv_2d\sigma(x)=0.\ee
For this purpose, we fix $U_1=\{x\in\partial\Omega:\ \nu(x)\cdot\eta_1\geq\epsilon\}$ and we remark that
$$\forall x\in U,\quad |\nu(x)-\eta_1|\geq |\nu(x)-\eta|-|\eta_1-\eta|\geq\epsilon,$$
which proves that $U\subset U_1$. Therefore, we have
$$\abs{\int_{U}\partial_\nu vv_2d\sigma(x)}\leq \int_{U_1}\abs{\partial_\nu vv_2}d\sigma(x).$$
Using the fact that $v_2$ takes the form \eqref{CGO2} and applying the Cauchy-Schwarz inequality
we get
\bel{t3g}\begin{aligned}\abs{\int_{U}\partial_\nu vv_2d\sigma(x)}&\leq \int_{U_1}e^{-\eta_1\cdot x}\abs{\partial_\nu v}(1+|w_2|)d\sigma(x)\\
\ &\leq C\left(\int_{U_1}e^{-2\rho\eta_1\cdot x}\abs{\partial_\nu v}^2d\sigma(x)\right)^{\frac{1}{2}}\left(1+\norm{w_2}_{L^2(\partial\Omega)}\right).\end{aligned}\ee
Recall that using local coordinates, one can check that there exists $C>0$ depending only on $\Omega$ such that
$$\norm{w_2}_{L^2(\partial\Omega)}\leq C\norm{w_2}_{L^2(\Omega)}^{\frac{1}{2}}\norm{\nabla w_2}_{L^2(\Omega)}^{\frac{1}{2}}\leq C\norm{w_2}_{L^2(\Omega)}^{\frac{1}{2}}\norm{ w_2}_{H^2(\Omega)}^{\frac{1}{2}}.$$
Therefore, applying \eqref{CGO3}, we deduce that there exists a constant $C>0$ independent of $\rho>0$ such that
$\norm{w_2}_{L^2(\partial\Omega)}\leq C$. Combining this with \eqref{t3g}, we obtain
\bel{t3h}\begin{aligned}\abs{\int_{U}\partial_\nu vv_2d\sigma(x)}\leq C\left(\int_{U_1}e^{-2\rho\eta_1\cdot x}\abs{\partial_\nu v}^2d\sigma(x)\right)^{\frac{1}{2}}.\end{aligned}\ee
On the other hand, using the fact that
$$\forall x\in U_1,\quad \nu(x)\cdot\eta_1\geq\epsilon,$$
we obtain
$$\begin{aligned}\int_{U_1}e^{-2\rho\eta_1\cdot x}\abs{\partial_\nu v}^2d\sigma(x)&\leq \int_{U_1}e^{-2\rho\eta_1\cdot x}\abs{\partial_\nu v}^2\frac{\eta_1\cdot\nu(x)}{\epsilon}d\sigma(x)\\
\ &\leq \epsilon^{-1}\int_{U_1}e^{-2\rho\eta_1\cdot x}\abs{\partial_\nu v}^2|\eta_1\cdot\nu(x)|d\sigma(x)\\
\ &\leq \epsilon^{-1}\int_{\partial\Omega_{+,\eta_1}}e^{-2\rho\eta_1\cdot x}\abs{\partial_\nu v}^2|\eta_1\cdot\nu(x)|d\sigma(x).\end{aligned}$$
Therefore, using the fact that $v\in H^1_0(\Omega)\cap H^2(\Omega)$ we can apply the Carleman estimate of Proposition \ref{p2}, with $q=q_2$,  to obtain
$$\begin{aligned}&\int_{U_1}e^{-2\rho\eta_1\cdot x}\abs{\partial_\nu v}^2d\sigma(x)\\
&\leq \rho^{-1}\epsilon^{-1}\left(\rho\int_{\partial\Omega_{+,\eta_1}}e^{-2\rho\eta_1\cdot x}\abs{\partial_\nu v}^2|\eta_1\cdot\nu(x)|d\sigma(x)\right)\\
&\leq C\rho^{-1}\epsilon^{-1}\left(\int_{\Omega}e^{-2\rho\eta_1\cdot x}\abs{-\Delta v+q_2v}^2dx+\rho\int_{\partial\Omega_{-,\eta_1}}e^{-2\rho\eta_1\cdot x}\abs{\partial_\nu v}^2|\eta_1\cdot\nu(x)|d\sigma(x)\right).\end{aligned}$$
Applying the fact that $\partial\Omega_{-,\eta_1}\subset V$, $-\Delta v+q_2v=(q_1-q_2)v_1$  and \eqref{t3d}, we get
$$\int_{U_1}e^{-2\rho\eta_1\cdot x}\abs{\partial_\nu v}^2d\sigma(x)\leq \left(\norm{q_1}_{L^\infty(\Omega)}+\norm{q_2}_{L^\infty(\Omega)}\right)C\rho^{-1}\epsilon^{-1}\int_{\Omega}e^{-2\rho\eta_1\cdot x}\abs{v_1}^2dx$$
and using \eqref{CGO1} and \eqref{CGO3}, we find
$$\begin{aligned}\int_{U_1}e^{-2\rho\eta_1\cdot x}\abs{\partial_\nu v}^2d\sigma(x)&\leq C\rho^{-1}\int_{\Omega}e^{-2\rho\eta_1\cdot x}\abs{e^{\rho\eta_1\cdot x}(1+|w_1|)}^2dx\\
\ &\leq C\rho^{-1},\end{aligned}$$
with $C$ independent of $\rho$. This proves \eqref{t3f}. Applying \eqref{t3f} and the arguments used in the proof of Theorem \ref{t2}, for $q=q_1-q_2$ extended by zero to a function of $\R^n$, we get 
\bel{t3q}\int_{\R^n} q(x)e^{-ix\cdot\xi}dx=0.\ee
Note that, following the above arguments, this property is true for any $\xi\in\eta_1^\bot$ while $\eta_1\in\{y\in\mathbb S^{n-1}:\ |y-\eta_1|<\epsilon\}$. Let us consider the following property.
\begin{lemma}\label{l4} Let $h\in L^1(\R^n)$ be compactly supported and fix the sequence $(\theta_k)_{k\in\mathbb N}$ of $\mathbb S^{n-1}$ such that 
\bel{l4a} \forall k,\ell\in\mathbb N,\quad  (k\neq\ell\Longrightarrow \theta_k\neq\theta_\ell).\ee
Then the condition
\bel{l4b}\forall k\in\mathbb N,\ \forall \xi\in \theta_k^{\bot},\quad \int_{\R^n} h(x)e^{-ix\cdot\xi}dx=0\ee
implies that $h=0$.\end{lemma}

We postpone the proof of Lemma \ref{l4} to the end of the present demonstration. According to Lemma \ref{l4}, the fact that \eqref{t3q} is fulfilled for all $\xi\in\eta_1^\bot$ for $\eta_1\in\{y\in\mathbb S^{n-1}:\ |y-\eta_1|<\epsilon\}$ implies that $q=0$. This proves that $q_1=q_2$ which completes the 
proof of Theorem \ref{t3}.\qed

Now that the proof of Theorem \ref{t3} is completed, let us consider Lemma \ref{l4}.

\textbf{Proof of Lemma \ref{l4}.} Fix $R>0$ such  that supp$(h)\subset\{y\in\R^n:\ |y|<R\}$. Note first that for all $\xi\in\R^n$, we have
$$\mathcal Fh(\xi)=\int_{\R^n}h(x)e^{-ix\cdot\xi}dx=\int_{|x|<R}h(x)\sum_{k=0}^\infty\frac{(-ix\cdot\xi)^k}{k!}dx.$$
Using the fact that
$$\forall x\in\{y\in\R^n:\ |y|<R\},\  \forall \xi\in\R^n,\quad \abs{\frac{(-ix\cdot\xi)^k}{k!}}\leq \frac{(R|\xi|)^k}{k!}$$
and 
$$\sum_{k=0}^\infty \frac{(R|\xi|)^k}{k!}=e^{R|\xi|}<\infty,$$
we deduce that
\bel{l4c}\forall \xi\in\R^n,\quad \mathcal Fh(\xi)=\sum_{k=0}^\infty\left(\int_{|x|<R}h(x)\frac{(-ix\cdot\xi)^k}{k!}dx\right).\ee
Combining this representation with \eqref{l4a}-\eqref{l4b}, we can show by iteration that for all $\ell\geq1$ we can find a complex valued real-analytic function $g_\ell:\R^n\longrightarrow \mathbb C$ such that
\bel{l4d}\forall \xi\in\R^n,\quad\mathcal Fh(\xi)=\left[\prod_{j=1}^\ell (\xi \cdot\theta_j)\right]g_\ell(\xi).\ee
Now let us fix $\xi\in\R^n$ and consider the function $P:\R\ni t\mapsto \mathcal Fh(t\xi)$. It is clear that
\bel{l4e}P(t)=\sum_{k=0}^\infty\left(\int_{|x|<R}h(x)\frac{(-ix\cdot\xi)^k}{k!}dx\right)t^k.\ee
Moreover, applying \eqref{l4d}, we deduce that
$$\begin{aligned}\forall \ell\geq 1,\ P(t)&=\left[\prod_{j=1}^\ell (t\xi \cdot\theta_j)\right]g_\ell(t\xi)\\
\ &=\left[\prod_{j=1}^\ell (\xi \cdot\theta_j)\right]g_\ell(t\xi)t^\ell=\underset{t\to0}{\mathcal O}(t^\ell).\end{aligned}$$
Combining this with \eqref{l4e} we obtain
$$\forall k\geq0,\quad \int_{|x|<R}h(x)\frac{(-ix\cdot\xi)^k}{k!}dx=0$$
and combining this with \eqref{l4c} we deduce that $\mathcal Fh(\xi)=0$. Since $\xi$  is arbitrary chosen we get
$$\forall \xi\in\R^n,\quad \mathcal Fh(\xi)=0$$
and using the injectivity of the Fourier transform we deduce that $h=0$. This completes the proof of the lemma.\qed
\subsection{Partial data result for  coefficients known close to the boundary}

In this subsection we consider an improvement of Theorem \ref{t3} for coefficients known on the neighborhood of the boundary of $\Omega$. The approach that we consider is borrowed from \cite{AU}. We fix $\gamma_1$, $\gamma_2$ two open sets of $\partial\Omega$ and we consider the subspace $H^{\frac{3}{2}}_{\gamma_1}(\partial\Omega)$ of $H^{\frac{3}{2}}(\partial\Omega)$ given by
$$H^{\frac{3}{2}}_{\gamma_1}(\partial\Omega):=\{h\in H^{\frac{3}{2}}(\partial\Omega):\ \textrm{supp}(h)\subset \gamma_1\}.$$
For $q\in\mathcal Q$, we consider the partial Dirichlet-to-Neumann map 
$$\Lambda_{\gamma_1,\gamma_2,q}:H^{\frac{3}{2}}_{\gamma_1}(\partial\Omega)\ni \phi\mapsto \partial_\nu v_{q,\phi}|_{\gamma_2},$$
with $v_{q,\phi}$ the solution of \eqref{eq2}. Then, we consider the following result.

\begin{theorem}
\label{t4}
Let $\Omega$ be connected. For $j=1,2$, let $q_j \in \mathcal Q$ be such that there exists an open set $\Omega_1\subset\Omega$, corresponding to a connected neighborhood of the boundary (i.e $\partial\Omega\subset\partial\Omega_1$), such that
\bel{t4a} q_1(x)=q_2(x),\quad x\in\Omega_1.\ee
Then we have
\bel{t4b} \Lambda_{\gamma_1,\gamma_2,q_1}=\Lambda_{\gamma_1,\gamma_2,q_2}\Longrightarrow q_1=q_2.\ee

 \end{theorem}

\begin{remark} In contrast to Theorem \ref{t3}, Theorem \ref{t4} is stated with arbitrary  restriction on the support of the Dirichlet input and on the location of the Neumann measurements. For instance, we can consider here problems with excitations and measurements on disjoint sets which is an interesting setting for several applications. However, in contrast to Theorem \ref{t3}, Theorem \ref{t4} requires \eqref{t4a} to be fulfilled which corresponds to the knowledge of the coefficients under consideration on a neighborhood of the boundary.

\end{remark}

For $\gamma$ an open set of $\partial\Omega$ and $q \in \mathcal Q$ we introduce the following sets
$$S_q:=\{v\in H^2(\Omega):\ -\Delta v+qv=0\},\quad S_{q,\gamma}:=\{v\in S_q:\ \textrm{supp}(v_{|\partial\Omega})\subset \gamma\}.$$
For Theorem \ref{t4}, we need to consider first the following density result.

\begin{lemma}\label{l33} The space $\{v_{|\Omega\setminus\Omega_1}:\ v\in S_{q,\gamma}\}$ is dense in $\{v_{|\Omega\setminus\Omega_1}:\ v\in S_{q}\}$ in the sense of $L^2(\Omega\setminus\Omega_1)$.\end{lemma}
\begin{proof} Assume the contrary. Then, the Hahn Banach theorem implies that there exists a continuous linear form on $L^2(\Omega\setminus\Omega_1)$ vanishing on $\{v_{|\Omega\setminus\Omega_1}:\ v\in S_{q,\gamma}\}$ but non uniformly vanishing on $\{v_{|\Omega\setminus\Omega_1}:\ v\in S_{q}\}$. This is equivalent to the existence of $y\in L^2(\Omega\setminus\Omega_1)$ and $v_0\in S_q$ such that
\bel{l33a}\forall v\in S_{q,\gamma},\quad \left\langle y,v\right\rangle_{L^2(\Omega\setminus\Omega_1)}=0,\ee
\bel{l33b} \left\langle y,v_0\right\rangle_{L^2(\Omega\setminus\Omega_1)}=1.\ee
We extend $y$ by zero to $\Omega$ and we consider $w\in H^2(\Omega)$ solving
\bel{l33c}
\left\{
\begin{array}{ll}
-\Delta w+qw=y  & \mbox{in}\ \Omega ,
\\
w=0 &\mbox{on}\ \partial\Omega.
\end{array}
\right.
\ee
Applying \eqref{l33a}, we get
$$\forall v\in S_{q,\gamma},\quad \int_\Omega (-\Delta+q)w\overline{v}dx=\left\langle y,v\right\rangle_{L^2(\Omega\setminus\Omega_1)}=0$$
and, integrating by parts, for all $v\in S_{q,\gamma}$,  we obtain
$$\int_\Omega w\overline{(-\Delta+q)v}dx-\int_{\partial\Omega}\partial_\nu w \overline{v}d\sigma(x) =\int_\Omega (-\Delta+q)w\overline{v}dx=0.$$
It follows that
\bel{l33d}\forall v\in S_{q,\gamma},\quad \int_{\gamma}\partial_\nu w \overline{v}d\sigma(x)=0.\ee
Since $q \in \mathcal Q$, for any $\phi\in\mathcal C^\infty_0(\gamma)$ there exists a unique $v_{q,\phi}\in H^2(\Omega)$ solving \eqref{eq2}. Using the fact that $v_{q,\phi}\in S_{q,\gamma}$, we deduce from \eqref{l33d} that 
$$\forall \phi\in\mathcal C^\infty_0(\gamma),\quad \int_{\gamma}\partial_\nu w \overline{\phi}d\sigma(x)=\int_{\gamma}\partial_\nu w \overline{v_{q,\phi}}d\sigma(x)=0.$$
This proves that $\partial_\nu w_{|\gamma}=0$ and combining this with \eqref{l33c} and the fact that $y_{|\Omega_1}=0$, we deduce that
$$\left\{
\begin{array}{ll}
-\Delta w+qw=0  & \mbox{in}\ \Omega_1 ,
\\
w=\partial_\nu w=0 &\mbox{on}\ \gamma.
\end{array}
\right.$$
Therefore, applying results of unique continuation for elliptic equations (e.g. \cite[Theorem 1.1]{GL} and \cite[Theorem 1]{SS}), we obtain $w=0$ on $\Omega_1$. In particular, we have $w=\partial_\nu w=0$ on $\partial\Omega$.
Therefore, integrating by parts, we find
$$\left\langle y,v_0\right\rangle_{L^2(\Omega\setminus\Omega_1)}=\int_\Omega (-\Delta+q)w\overline{v_0}dx=\int_\Omega w\overline{(-\Delta+q)v_0}dx=0,$$
which contradicts \eqref{l33b}. This completes the proof of the lemma.\end{proof}
Armed with this lemma, we are now in position to complete the proof of Theorem \ref{t4}.

\textbf{Proof of Theorem \ref{t4}.} For $j=1,2$, consider $v_j\in S_{q_j,\gamma_j}$. We consider $y_2\in H^2(\Omega)$ solving
$$\left\{
\begin{array}{ll}
-\Delta y_2+q_2y_2=0  & \mbox{in}\ \Omega ,
\\
y_2=v_1 &\mbox{on}\ \partial\Omega.
\end{array}
\right.$$
Then, $v=y_2-v_1$ solves
\bel{t4d}\left\{
\begin{array}{ll}
-\Delta v+q_2v=(q_1-q_2)v_1  & \mbox{in}\ \Omega ,
\\
v=0 &\mbox{on}\ \partial\Omega.
\end{array}
\right.\ee
Moreover, fixing $\phi\in H^{\frac{3}{2}}_{\gamma_1}(\partial\Omega)$ defined by $\phi(x)=v_1(x),\ x\in\partial\Omega,$
we find
$$\partial_\nu v_{|\gamma_2}=\Lambda_{q_2,\gamma_1,\gamma_2}\phi-\Lambda_{q_1,\gamma_1,\gamma_2}\phi=0.$$
Therefore, multiplying \eqref{t4d} by $v_2$ and integrating by parts we get
$$\begin{aligned}\int_\Omega (q_1-q_2)v_1v_2dx &=\int_\Omega (-\Delta v+q_2v)v_2dx\\
\ &=\int_\Omega v(-\Delta v_2+q_2v_2)dx-\int_{\gamma_2}\partial_\nu vv_2d\sigma(x)\\
\ &=0.\end{aligned}$$
Applying \eqref{t4a}, we deduce that, for any $v_j\in S_{q_j,\gamma_j}$, $j=1,2$, we have
$$\int_{\Omega\setminus\Omega_1} (q_1-q_2)v_1v_2dx=\int_\Omega (q_1-q_2)v_1v_2dx=0.$$
Using the density result of Lemma \ref{l33}, we deduce that, for any $v_j\in S_{q_j}$, $j=1,2$, we have
$$\int_\Omega (q_1-q_2)v_1v_2dx=\int_{\Omega\setminus\Omega_1} (q_1-q_2)v_1v_2dx=0.$$
Thus, choosing $v_j$, $j=1,2$, of the form \eqref{CGO1}-\eqref{CGO2}, with $w_j\in H^2(\Omega)$ satisfying \eqref{CGO3}, in this last identity and repeating the argument used at the end of the proof of Theorem \ref{t2}, we deduce that $q_1=q_2$.\qed

\bigskip
\vskip 1cm


\begin{thebibliography}{CS2}
 \frenchspacing

 
\bibitem[AU]{AU} {\sc H. Ammari and G. Uhlmann}, {\em Reconstuction from partial Cauchy data for the
Schr\"odinger equation}, Indiana University Math J., \textbf{53} (2004), 169-184.
\bibitem[BU]{BU}{\sc A. L. Bukhgeim and G. Uhlmann}, {\em Recovering a potential from partial  Cauchy data}, Commun. Partial Diff. Eqns., \textbf{27} (2002), no 3-4, 653-668.
\bibitem[C]{C}{\sc A. Calder\'on}, {\em On an inverse boundary problem}, Seminar on Numerical Analysis and its Applications to Continuum Physics, Soc. Brasileira de Matem\'atica, Rio de Jane\'{i}ro (1980), 65-73.

\bibitem[Ch]{Ch}{\sc M. Choulli}, {\em Une introduction aux probl\`emes inverses elliptiques et paraboliques}, Math\'ematiques et Applications, Vol. 65, Springer-Verlag, Berlin, 2009.
\bibitem[GL]{GL}{\sc N. Garofalo and F-H. Lin}, {\em Unique continuation for elliptic operators: a  geometric-variational approach}, Communications on Pure and Applied Mathematics, \textbf{40} (1987), 347-366.
\bibitem[Ha]{Ha} {\sc P. H\"ahner}, {\em A periodic Faddeev-type operator}, J. Diff. Equat., {\bf 128} (1996), 300-308.
\bibitem[J]{J} {\sc J. Jossinet}, {\em The impedivity of freshly excised human breast tissue}, Physiol. Meas., \textbf{19} (1998), 61-75.
 \bibitem[Ka]{Ka}{\sc O. Kavian}, {\em Four Lectures on Parameter Identification}, Three Courses on Partial Differential Equations, pp. 125-162, IRMA Lect. Math. Theor. Phys., 4, de Gruyter, Berlin, 2003.
 \bibitem[Ki]{Ki}{\sc Y. Kian}, {\em Recovery of non compactly supported coefficients of elliptic equations on an infinite waveguide}, to appear in Journal of the Institute of Mathematics of Jussieu, 	http://dx.doi.org/10.1017/S1474748018000488.
 \bibitem[KV]{KV}{\sc R.V. Kohn and M. Vogelius} {\em Determining conductivity by boundary measurements},
Comm. Pure Appl. Math., \textbf{37} (1984), pp. 289-297.
\bibitem[SS]{SS} {\sc J. C. Saut and B. Scheurer}, {\em Sur l'unicit\'e du probl\`eme de Cauchy et le prolongement
unique pour des \'equations elliptiques \`a coefficients non localement
born\'es}, J. Diff. Equat., \textbf{43} (1982), 28-43.
\bibitem[SU]{SU} {\sc J. Sylvester and G. Uhlmann}, {\em A global uniqueness theorem for an inverse boundary value problem}, Ann. of Math., \textbf{125} (1987), 153-169.
\bibitem[Uh]{Uh}{\sc G. 	Uhlmann}, {\em  Electrical impedance tomography and Calder\'on's problem}, Inverse problems,  \textbf{25} (2009), 123011.
\bibitem[ZK]{Z}{\sc M. S. Zhdanov and  G. V. Keller}, {\em The geoelectrical methods in geophysical exploration}, Methods in Geochemistry and Geophysics, \textbf{31} (1994), Elsevier.
 
\bibliographystyle{abbrv}

\end{thebibliography}
\end{document}